\documentclass[12pt,a4paper]{amsart}
\usepackage{amsfonts}
\usepackage{amsmath}
\usepackage{mathrsfs}
\usepackage{graphicx}
\usepackage{amsthm}
\usepackage{amssymb}
\usepackage{amscd}
\usepackage{color}
\usepackage{enumerate}
\usepackage{epsfig}
\usepackage{hyperref}
\usepackage{comment}
\usepackage{mathtools}
\usepackage{paralist}

\usepackage[T1]{fontenc}

\numberwithin{equation}{section}
\theoremstyle{plain}\newtheorem{maintheorem}{Theorem}

\newtheorem{Thm}{Theorem}[section]

\theoremstyle{remark}
\newtheorem{Def} {Definition}[section]

\newtheorem{lem} {Lemma}[section]
\newtheorem{prop} {Proposition}[section]

\theoremstyle{remark}
\newtheorem{rem}  {Remark}[section]

\theoremstyle{definition}

\theoremstyle{definition}

\theoremstyle{remark}

\numberwithin{equation}{section}

\newcommand{\N}{{\mathbb{N}}}

\newcommand{\B }{{\mathcal{B}}}

\newcommand{\cJ }{{\mathcal{J}}}
\newcommand*{\dif}{\mathop{}\!\mathrm{d}}

\begin{document}
\title[Points that are Recurrent  but Not Banach Recurrent] {Existence and Distributional  Chaos of Points that are Recurrent  but Not Banach Recurrent }
\author[Y. Jiang] {Yifan Jiang}
\address[Y. Jiang]{Mathematical Institute, University of Oxford  \\  Oxford OX2 6GG, United Kingdom}
\email{yifan.jiang@maths.ox.ac.uk}
\urladdr{https://www.maths.ox.ac.uk/people/yifan.jiang}

\author[X. Tian] {Xueting Tian}
\address[X. Tian]{School of Mathematical Sciences,  Fudan University   \\  Shanghai 200433, People's Republic of China}
\email{xuetingtian@fudan.edu.cn}
\urladdr{http://homepage.fudan.edu.cn/xuetingtian}

\thanks{Y. Jiang was partially supported by Fudan project FDUROP (Fudan's Undergraduate Research Opportunities Program), and the EPSRC Centre for Doctoral Training in Mathematics of Random Systems: Analysis, Modelling, and Simulation (EP/S023925/1). X. Tian is the corresponding author and was supported by National Natural Science Foundation of China  (grant No.12071082) and in part by Shanghai Science and Technology Research Program (grant No.21JC1400700).}

\def\abstractname{\textbf{Abstract}}

\begin{abstract}
  \addcontentsline{toc}{section}{\bf{English Abstract}}
  According to the recurrent frequency, many levels of recurrent points are found such as periodic points, almost periodic points,  weakly almost periodic points, quasi-weakly almost periodic points and Banach recurrent points.
  In this paper, we consider symbolic dynamics and show the existence of six refined levels between Banach recurrence and general recurrence.
  Despite the fact that these refined levels are all null-measure under any invariant measure, we further show they carry strong topological complexity.
  Each refined level of recurrent points is dense in the whole space and contains an uncountable distributionally chaotic subset.
  For a wide range of dynamical systems such as expansive systems with the shadowing property, we also show the distributional chaos of the points that are recurrent but not Banach recurrent.
  
\end{abstract}
\maketitle

\section{Introduction}
\label{intro}
In classical study of dynamical systems, one fundamental problem is to consider the asymptotic behavior of the orbits.
According to the `recurrent frequency'  (i.e.,  the probability of finding the orbit of an initial point entering in its neighborhood), many levels of recurrent points  have been studied such as periodic points, almost periodic points,  weakly almost periodic points, quasi-weakly almost periodic points and Banach recurrent points, see \cite{HYZ,ZH}.
The existence and topological complexity of different recurrent levels are always of interest.
It is known that the set of recurrent points has full measure for any invariant measure by Poincar\'e recurrence theorem  and the set of minimal points  is not empty \cite{Furst,Walter}.
For systems with the Bowen specification property (such as topological mixing subshifts of finite type and topological mixing uniformly hyperbolic systems), the set of periodic points is dense in the whole space \cite{Sigmund1974}.
In \cite{T16,HTW,CT2019} the authors considered various Banach recurrence and showed different recurrent levels carry strong dynamical complexity from the perspective of topological entropy and distributional chaos.

In this paper we aim to study the existence and complexity of points that are recurrent but not Banach recurrent.
Let \((X,d)\) be a compact metric space and \(f:X \to X\) be a continuous map.
Such \((X,f)\) is called a (topological) dynamical system.
For any \(x \in X\), the orbit of \(x\) is \(\{f^n(x)\}_{n=0}^\infty\), denoted by \(orb(x)\).
For any \(\varepsilon>0\), \(B_{\varepsilon}(x)\) denotes the open ball with radius \(\varepsilon\) centered at \(x\).
Let \(U,\, V\subseteq X\) be two nonempty open subsets and \(x\in X\).
We define the sets of visiting times
\[N (U, V):=\{n\geq 1\,|\,  U\cap f^{-n} (V)\neq \emptyset\} \,  \,  \text{ and } \,  \,  N (x,  U):=\{n\ge 1|\,  f^n (x)\in U\}.\]
For a subset  \(S\subseteq \mathbb{N}\),  we define four different densities of \(S\) as the following
\[\bar{d} (S):=\limsup_{n\rightarrow\infty}\frac{|S\cap \{0,  1,  \cdots,  n-1\}|}n,  \,  \,   \underline{d} (S):=\liminf_{n\rightarrow\infty}\frac{|S\cap \{0,  1,  \cdots,  n-1\}|}n, \]
\[B^* (S):=\limsup_{|I|\rightarrow\infty}\frac{|S\cap I|}{|I|},  \,  \,   B_* (S):=\liminf_{|I|\rightarrow\infty}\frac{|S\cap I|}{|I|}.\]
Here \(I\subseteq \mathbb{N}\) is taken from finite consecutive integer intervals and \(|\cdot|\) denotes the cardinality of a set.
These four concepts are called {\it upper density} and {\it lower density} of \(S\), {\it Banach upper density} and {\it Banach lower density} of \(S\),   respectively, which are basic and have played important roles in the field of dynamical systems,   ergodic theory and number theory,   etc.

We first introduce six levels of recurrent points arranged according to their asymptotic behaviors, see \cite{HTW,HYZ,Walter,ZH}.
\begin{Def}
  \label{def-recurr}
  A point \(x\in X\) is \textit{regular}, if empirical measures \(\frac{1}{n}\sum_{i=0}^{n-1}\delta_{f^i(x)}\) converge to an ergodic measure \(\mu_{x}\) and \(x\) is contained in the support of \(\mu_{x}\);\\
  A point \(x\in X\) is \textit{weakly almost periodic}, if \(N(x,B_{\varepsilon}(x))\) has a positive lower density for any \(\varepsilon>0\);\\
  A point \(x\in X\) is \textit{quasi-weakly almost periodic}, if \(N(x,B_{\varepsilon}(x))\) has a positive upper density for any \(\varepsilon>0\);\\
  A point \(x\in X\) is \textit{Banach recurrent}, if \(N(x,B_{\varepsilon}(x))\) has a positive Banach upper density for any \(\varepsilon>0\);\\
  A point \(x\in X\) is \textit{recurrent}, if \(N(x,B_{\varepsilon}(x))\neq\emptyset\) for any \(\varepsilon>0\);\\
  A point \(x\in X\) is \textit{non-wandering}, if \(N(B_{\varepsilon}(x),B_{\varepsilon}(x))\neq\emptyset\) for any \(\varepsilon>0\).
\end{Def}
The set of points with above asymptotic behaviors are denoted by \(R(X),\) \(W(X),\) \(QW(X),\) \(BR(X),\) \(Rec(X),\) \(\Omega(X)\) respectively.
We will omit the argument \(X\) when the dynamical system is clear.
From the above definitions, we find that these six recurrent levels are invariant and actually nested from left to right.
The set of regular points has totally full measure \cite[Theorem 2.5 on Page 120]{Oxt}, (i.e., carrying full measure under any invariant measure).
This implies the gap-sets between \(R,\) \(W,\) \(QW,\) \(BR,\) \(Rec,\) \(\Omega\) have totally zero measure, and thus they are negligible from the probabilistic perspective.
Surprisingly, these gap-sets may still carry strong  topological  complexity.
For example, the sets  \((W\setminus R),\, (QW\setminus W),\,(BR\setminus QW)  \text{ and }    (\Omega \setminus Rec)\)   have the same topological entropy as the system itself and display strong distributional chaos for a large class of systems including all transitive Anosov diffeomorphisms, see  \cite{T16,HTW,CT2019,DongTian2016} for details.
However, the topological complexity characterizations and even the existence of  \((Rec\setminus BR)\) are still unknown.
Our first result shows that, for one-sided full shifts, the gap-set \((Rec\setminus BR)\) is dense and carries distributional chaos of type 1.
\begin{maintheorem}
  \label{Thm-main-symbolic}
  Given \(k\geq 2\), let \((\Sigma,\sigma)\) denote the one-sided full shift on \(k\) symbols.
  Then the gap-set  \(Rec(\Sigma)\setminus BR(\Sigma)\)  is dense in \(\Sigma\) and contains an uncountable DC-1 scrambled subset.
\end{maintheorem}

Motivated by using different densities to describe the probability of one orbit entering one set, several statistical \(\omega\)-limit sets   were introduced in \cite{DongTian2016}.
These concepts describe the place where one general orbit can go and then give many refined characterizations of general orbits.
\begin{Def}
  The \(\omega\)-limit set of \(x\) is the set of all limit points of \(orb(x)\), denoted by \(\omega_{f}(x)\).
  For \(x\in X\) and \(\xi=\overline{d},   \,  \underline{d},   \,  B^*,  \,   B_*\),   a point \(y\in X\) is called \(x\)-\(\xi\)-accessible,   if \(N (x,  B_\varepsilon (y))\) has a positive density with respect to \(\xi\) for any \(\varepsilon>0\).
  Then, we define the \(\omega_{\xi}\)-limit set of \(x\) as
  \[\omega_{\xi}(x):=\{y\in X\,  |\,   y\text{ is } x\text{-}\xi\text{-accessible}\}.\]
\end{Def}
With these definitions, one can immediately obtain that
\begin{equation*}\label{omega-order}
  \omega_{B_*}(x)\subseteq \omega_{\underline{d}}(x)\subseteq \omega_{\overline{d}}(x)\subseteq \omega_{B^*}(x)\subseteq { \omega_{f}(x)}.
\end{equation*}
For any \(x\in X\), if \(\omega_{B_*}(x)=\emptyset\), then we know that  \(x\) satisfies one and only one of following twelve cases:
\begin{description}
  \item[Case  (1)\;  ] \,  \, \(\omega_{B_*}(x)\subsetneq\omega_{\underline{d}}(x)= \omega_{\overline{d}}(x)= \omega_{B^*}(x)= { \omega_{f}(x)};\)
  \item[Case  (1')   ]  \,  \,   \(\omega_{B_*}(x)\subsetneq\omega_{\underline{d}}(x)= \omega_{\overline{d}}(x)= \omega_{B^*}(x)\subsetneq{ \omega_{f}(x)};\)
  \item[Case  (2)\;  ]  \,  \,   \(\omega_{B_*}(x)\subsetneq    \omega_{\underline{d}}(x)=\omega_{\overline{d}}(x)  \subsetneq \omega_{B^*}(x)= { \omega_{f}(x)} ;\)
  \item[Case  (2')   ]  \,  \,   \(\omega_{B_*}(x)\subsetneq  \omega_{\underline{d}}(x)=\omega_{\overline{d}}(x)  \subsetneq \omega_{B^*}(x)\subsetneq { \omega_{f}(x)} ;\)
  \item[Case  (3)\;  ]  \,  \,   \(\omega_{B_*}(x)=\omega_{\underline{d}}(x)\subsetneq\omega_{\overline{d}}(x)= \omega_{B^*}(x)= { \omega_{f}(x)} ;\)
  \item[Case  (3')   ] \,  \,   \(\omega_{B_*}(x)=\omega_{\underline{d}}(x)\subsetneq\omega_{\overline{d}}(x)= \omega_{B^*}(x)\subsetneq { \omega_{f}(x)} ;\)
  \item[Case  (4)\;  ]   \,  \,  \(\omega_{B_*}(x)\subsetneq  \omega_{\underline{d}}(x)\subsetneq\omega_{\overline{d}}(x)= \omega_{B^*}(x)= { \omega_{f}(x)} ;\)
  \item[Case  (4')   ]  \,  \, \(\omega_{B_*}(x)\subsetneq  \omega_{\underline{d}}(x)\subsetneq\omega_{\overline{d}}(x)= \omega_{B^*}(x)\subsetneq { \omega_{f}(x)} ;\)
  \item[Case  (5)\;  ]  \,  \,   \(\omega_{B_*}(x)=  \omega_{\underline{d}}(x) \subsetneq \omega_{\overline{d}}(x)\subsetneq \omega_{B^*}(x)= { \omega_{f}(x)};\)
  \item[Case  (5')   ] \,  \,   \(\omega_{B_*}(x)=  \omega_{\underline{d}}(x) \subsetneq \omega_{\overline{d}}(x)\subsetneq \omega_{B^*}(x)\subsetneq { \omega_{f}(x)};\)
  \item[Case  (6)\;  ]   \,  \, \(\omega_{B_*}(x)\subsetneq \omega_{\underline{d}}(x) \subsetneq \omega_{\overline{d}}(x)\subsetneq \omega_{B^*}(x)= { \omega_{f}(x)};\)
  \item[Case  (6')   ]  \,  \,  \(\omega_{B_*}(x)\subsetneq \omega_{\underline{d}}(x) \subsetneq \omega_{\overline{d}}(x)\subsetneq \omega_{B^*}(x)\subsetneq { \omega_{f}(x)}.\)
\end{description}
There are twelve cases rather than sixteen because  \(\omega_{\overline{d}}(x)\) must be a nonempty set, see \cite{DongTian2016}.
We point out that a recurrent point  satisfying one of Cases (1)-(6) is Banach recurrent.
Thus, the set \((Rec\setminus BR)\) can only be contained in Cases (1')-(6').
Now we give a refined result of Theorem \ref{Thm-main-symbolic}.
\begin{maintheorem}
  \label{Thm-main-omega-DC1}
  Given \(k\geq 2\), let \((\Sigma,\sigma)\) be the one-sided full shift on \(k\) symbols.
  Then \(\left\{ x\in Rec(\Sigma)\setminus BR(\Sigma) \,|\,x \text{ satisfies Case (i') }\right\}\) is dense in \(\Sigma\) and contains an uncountable DC-1 scrambled subset for i=1-6.
\end{maintheorem}

The technique used in this paper is to construct a map \(T\) which maps Case (i) to Case (i') and keeps distributional chaos invariant.
The remainder of the paper is organized as follows.
In section 2, we introduce basic concepts and notations of symbolic dynamics, distributional chaos and multifractal analysis.
Section 3 discusses the construction of \(T\) and its properties.
In section 4, we construct an auxiliary subshift which satisfies the specification property.
Section 5 is devoted to the proof of main theorems.
In section 6, we extend the previous results to general dynamical systems.

We remark that the topological entropy on the set of  \((Rec\setminus BR)\)  is still an open problem.
The map \(T\) constructed in this paper does not necessary map a full entropy set to a full entropy set.

\section{Preliminaries}
\label{sec-prem}
\subsection{Basic notations and facts}
\label{subsec-not}
Let \((X,d)\) be a compact metric space and \(f:X\to X\) be a continuous map.
Let \( M(X)\) denote  the space of probability measures  on \(X\) equipped with the weak-star topology, \( M_{f}(X)\) the set of all invariant measures under \(f\).
For any measure \(\mu\in M(X)\), by \(S_\mu\) we denote its support.
By \(V_{f}(x)\) we denote the limit points of the empirical measures \(\left\{\frac{1}{n}\sum_{i=0}^{n-1}\delta_{f^i(x)}\right\}_{n\geq 1}\).

The following characterization of statistical \(\omega\)-limit sets is from \cite{DongTian2016}.
\begin{lem}\label{Lem-omegalimit}
  For any \(x\in X\) we have \(\omega_{\underline{d}}(x)=\bigcap_{\mu\in V_{f}(x)} S_\mu\) and \(\omega_{\overline{d}}(x)=\overline{\bigcup_{\mu\in V_{f}(x)} S_\mu}\).
  Furthermore, we have \(\omega_{B_*}(x)=\bigcap_{\mu\in  M_{f}(\Lambda)} S_\mu\) and \(\omega_{B^*}(x)=\overline{\bigcup_{\mu\in  M_{f}(\Lambda)} S_\mu}\), where \(\Lambda=\omega_{f}(x)\).
\end{lem}

A point \(x\in X\) is \textit{transitive} if \(\omega_{f}(x)=X\) and we denote all the transitive points in \(X\) by \(Trans(X)\).
A dynamical system \((X,f)\) is \textit{transitive} if for every pair of nonempty open sets \(U,\,V\subseteq X\), there exists a positive integer \(n\) such that \(f^{n}(U)\cap V\neq \emptyset.\)
A dynamical system \((X,f)\) is \textit{topological mixing}  if there exists a positive integer \(N\) such that \(f^{n}(U)\cap V\neq \emptyset\) for any \(n\geq N\).

Following \cite{PM,CT2019}, we introduce the specification property which was firstly introduced by Bowen.
For any \(\varepsilon>0\), \(x,y\in X\) and nonnegative integers \(n\leq m\), we say \(x\) \(\varepsilon\)\textit{-traces} \(y\) \textit{on} \([n,m]\) if for any \(n\leq i\leq m\)
\[d(f^{i}(x),f^{i-n}(y))<\varepsilon.\]
A dynamical system \((X,f)\) has \textit{the specification property} if for any \(\varepsilon>0\), there is a positive integer \(N_{\varepsilon}\) such that for any integer \(s\geq 2\), any \(s\) points \(\{y_{1},\cdots,y_{s}\}\) in \(X\), and any sequence
\[0=n_{1}\leq m_{1}<n_{2}\leq m_{2}<\cdots <n_{s}\leq m_{s}\]
of \(2s\) integers with
\[n_{i+1}-m_{i}\geq N_{\varepsilon}\]
for \(1\leq i<s\), there exists a point \(x\in X\) such that \(x\) \(\varepsilon\)-traces \(y_{i}\) on \([n_{i},m_{i}]\) for all \(1\leq i\leq s\).

\subsection{Symbolic dynamics}
\label{subsec-symb}
We introduce the standard one-sided full shift as follows.
Given a positive integer \(k\), let \(\mathcal{A}=\left\{0,1,\cdots,k-1\right\}\) be a set of \(k\) symbols.
Let \(\Sigma=\prod_{n=1}^{\infty}\mathcal{A}\).
We denote each point in \(\Sigma\) by a lowercase letter, e.g., \(a=(a_{1}a_{2}\cdots)\).
Then \(\Sigma\) is a compact space equipped with the following metric
\[d(a,b):=\sum_{n=1}^{\infty}\frac{\delta(a_n,b_n)}{2^n},\  \text{for any } a,\, b\in \Sigma.\]
The shift operator \(\sigma:\Sigma\to\Sigma\) is defined as
\[\sigma(a)=(a_{2}a_{3}\cdots),\  \text{for any }a=(a_1a_2\cdots)\in \Sigma.\]
By \((\Sigma,\sigma)\) we denote the one-sided full shift on \(k\) symbols.
We say \((\Pi,\sigma)\) is a subshift if \(\Pi\) is closed and invariant under \(\sigma\).

We use an uppercase letter to denote a  word in \(\Sigma\), e.g, \(A=(a_{1}\cdots a_{n})\) and by \(|A|\) we denote the length of \(A\).
For any  word \(A\), let \(f_{n}(A)\) denote the first \((n\wedge|A|)\) symbols of \(A\), \(l_{n}(A)\) denote the last \((n\wedge|A|)\) symbols of \(A\),  and \([A]\) denote the cylinder set
\(\{x\in\Sigma\,|\,A=X_{|A|}\}.\)
For any two words \(A=(a_1\cdots a_n)\) and \(B=(b_1\cdots b_m)\), we define the concatenation  of \(A\) and \(B\) as \(AB=(a_1\cdots a_nb_1\cdots b_m)\).
For any point \(a=(a_{1}a_{2}\cdots)\) in \(\Sigma\), we denote the first \(n\) symbols of \(a\) by the corresponding uppercase letter \(A_{n}=(a_{1}\cdots a_{n})\).
We say a  word \(A=(a_1\cdots a_n)\) appears in a set \(S\subseteq \Sigma\), if there exists \(x\in S\) such that the permutation \((a_{1}\cdots a_{n})\) appears in \((x_{1}x_{2}\cdots)\).
A word \(A=(a_1\cdots a_n)\) appears in another word \(B=(b_1\cdots b_m)\) if the permutation \((a_{1}\cdots a_{n})\) appears in \((b_{1}\cdots b_{m})\).

\subsection{Distributional chaos}
\label{subsec-dc}
We will give a short review on distributional chaos which was introduced in \cite{SS1994,PM}.
Given a compact dynamical system \((X,f)\), for any positive integer \(n\), points \(x,\, y \in X\) and \(t \in [0,\mathrm{diam}(X)]\) define
\[\Phi_{x y}^{(n)}(t):=\frac{1}{n}\left|\left\{i \,|\,  d(f^{i}(x), f^{i}(y))<t, 0 \leq i<n\right\}\right|,\]
Let \(\Phi_{x y}\) and \(\Phi_{x y}^{*}\) be the
following functions:
\[\Phi_{x y}(t):=\liminf _{n \rightarrow \infty} \Phi_{x y}^{(n)}(t), \quad \Phi_{x y}^{*}(t):=\limsup _{n \rightarrow \infty} \Phi_{x y}^{(n)}(t).\]
Both functions \(\Phi_{x y}\) and \(\Phi_{x y}^{*}\) are nondecreasing.
\begin{Def}
  A pair of points \((x, y) \in X \times X\) is called \textit{distributionally chaotic of type 1} if
  \[{\Phi_{x y}(s)=0 \text { for some } s>0 \text { and }} \\ {\Phi_{x y}^{*}(t)=1 \text { for all } t>0}.\]
  A set containing at least two points is called a \textit{distributional chaos scrambled set of type 1} (DC1-scrambled, for short) if any pair of its distinct points is distributionally chaotic of type 1.
\end{Def}

\subsection{Multifractal analysis}
\label{subsec-ma}
We review the irregular set  and the regular set  \cite{Pesin-Pitskel1984,Barreira-Schmeling2000,DOT,Clim,TaV,Tho12,Ruelle,Oxt} of a dynamical system \((X,f)\) with respect to a continuous function \(\varphi\) on \(X\).
For a continuous function \(\varphi\) on \(X\), define the \(\varphi\)-irregular set as
\[
  I_{\varphi}(X):=\left\{x \in X\,\bigg| \, \lim _{n \rightarrow \infty} \frac{1}{n} \sum_{i=0}^{n-1} \varphi(f^{i} (x)) \text { diverges }\right\}.
\]
Denote
\[
  L_{\varphi}(X):=\left[\inf _{\mu \in \mathcal{M}_{f}(X)} \int \varphi \dif \mu, \sup _{\mu \in \mathcal{M}_{f}(X)} \int \varphi \dif \mu\right]
\]
and
\[
  \operatorname{Int}\left(L_{\varphi}(X)\right):=\left(\inf _{\mu \in \mathcal{M}_{f}(X)} \int \varphi \dif \mu, \sup _{\mu \in \mathcal{M}_{f}(X)} \int \varphi \dif \mu\right).
\]
For any \(a \in L_{\varphi}(f)\), define the level set as
\[
  R_{\varphi}(X,a):=\left\{x \in X \,\bigg|\, \lim _{n \rightarrow \infty} \frac{1}{n} \sum_{i=0}^{n-1} \varphi(f^{i} (x))=a\right\}.
\]
\(\text {Denote } R_{\varphi}(X)=\bigcup_{a \in L_{\varphi}(X)} R_{\varphi}(X,a)\), called the regular points for \(\varphi\).

In view of multifractal analysis, we state a finer result of Theorem \ref{Thm-main-omega-DC1}.
\begin{maintheorem}\label{Thm-multifraction}
  Given \(k\geq 2\), let \((\Sigma,\sigma)\) be the one-sided full shift on \(k\) symbols.
  Suppose that  \(\varphi\) is a continuous function on \(\Sigma\) and \(I_{\varphi}(\Sigma)\neq \emptyset\).
  Then for any \(a\in \operatorname{Int}(L_{\varphi}(\Sigma))\), the recurrent level set  \(\left\{ x\in Rec(\Sigma)\setminus BR(\Sigma)\,|\,x\ satisfies\ Case (i')\right\}\)  contains an uncountable DC1-scrambled subset in \(I_{\varphi}(\Sigma)\), \(R_{\varphi}(\Sigma)\) and \(R_{\varphi}(\Sigma,a)\), respectively, for i=1-6.
\end{maintheorem}

\section{Construction and properties of map \(T\)}
Assume \((\Pi,\sigma)\) is a proper subshift of \((\Sigma,\sigma)\).
In this section, we construct an injective map \(T:\Pi\to\Sigma\), which is important to show the existence and density of Case (i') for i=1-6.
We will show \(T(\Pi)\subseteq Rec(\Sigma)\setminus BR(\Sigma)\), and \(T\) maps any point in \(\Pi\) satisfying Case (i)  to a point  in \(\Sigma\) satisfying Case (i').
Meanwhile this map keeps many properties unchanged in the view of invariant measure, Birkhoff ergodic average and distributional chaos.

We pick a fixed word \(A^{1}\) appearing in \(\Sigma\) but not in \(\Pi\).
For any \(x\in\Sigma\), \(Tx\) is defined by induction as follows.
Let \[B^n=X_n\cdots X_n \quad \text{and}\quad A^{n+1}=A^nB^nA^n,\]
where \(X_n\) repeats \(\left|A^n\right|^2\) times.
We define
\[Tx:=\bigcap_{n\geq 1}[A^{n}],\]
since \([A^{n}]\) are nested closed sets and \(\lim_{n\to\infty}\mathrm{diam}([A^{n}])=0\).

We remark that this construction is mainly inspired by the constructions in \cite{Petersen,HZ}.

\subsection{The image of \(T\)}
We show that the image of \(T\) is contained in \(Rec\setminus BR\) first.
\begin{prop}
  \label{prop-BR}
  For any \(x\in \Pi\), we have \(\omega_{B^*}(x)\subseteq\omega_{B^*}(Tx)\subseteq \Pi \subsetneq \omega_{\sigma}(Tx)\).
  Moreover, we have \(Tx\in Rec(\Sigma)\setminus BR(\Sigma)\).
\end{prop}

\begin{proof}
  Let \(y=Tx\).
  By definition, it is easy to check \(\omega_{B^*}(x)\subseteq\omega_{B^*}(Tx)\).
  Then, we show \(\omega_{B^*}(y)\subseteq \Pi\) by contradiction.
  Suppose there exists \(z\in\omega_{B^*}(y)\setminus \Pi \neq \emptyset\).
  Since \(\Pi\) is a proper closed subset of \(\Sigma\), there exists \(N\) such that \(Z_{N}\) does not appear in \(\Pi\).
  Then, we define an index for any finite word \(A=(a_1\cdots a_n)\) as
  \[\pi(A):=|\left\{1 \leq i \leq n-N+1 \,|\,(a_{i}\cdots a_{i+N-1})=Z_{N}\right\}|.\]
  \(\pi(A)\) counts the number of appearances of \(Z_{n}\) in \(A\).
  It is easy to check that for any words \(A\) and \(B\)
  \[\pi(AB)\leq \pi(A)+\pi(B)+N.\]
  Since \(B^{i}=X_{i}\cdots X_{i}\) is the concatenation of several \(X_{i}\), we get for any \(i\geq N\)
  \begin{equation}
    \label{eqn-b}
    \sup_{n\geq 1}\frac{\pi(f_{n}(B^{i}))}{n}\leq \frac{N}{i} \quad \text{and}\quad \sup_{n\geq 1}\frac{\pi(l_{n}(B^{i}))}{n}\leq \frac{N}{i}.
  \end{equation}
  In particular, we have for \(i\geq N\)
  \begin{equation}
    \frac{\pi(B^{i})}{|B^{i}|}\leq \frac{N}{i}.
  \end{equation}

  We claim that
  \begin{equation}
    \label{eqn-claim}
    \limsup_{n\to\infty}\sup_{i\geq 1}\frac{\pi(f_{n}(A^{i}))}{n}=\limsup_{n\to\infty}\sup_{i\geq 1}\frac{\pi(l_{n}(A^{i}))}{n}=0.
  \end{equation}
  Let \(i_{n}\) satisfy \(|A^{i_{n}}|\leq n<|A^{i_{n}+1}|\).
  Since \(A^{n+1}=A^{n}B^{n}A^{n}\), we have
  \[\sup_{i\geq 1}\frac{\pi(f_{n}(A^{i}))}{n}=\frac{\pi(f_{n}(A^{i_{n}+1}))}{n}.\]
  Again, by noticing \[\pi(f_{n}(A^{i_{n}+1}))\leq2\pi(A^{i_{n}})+\pi(f_{n}(B^{i_{n}}))+2N,\] we obtain
  \begin{align}
    \label{eqn-11}
    \nonumber\frac{\pi(f_{n}(A^{i_{n}+1}))}{n} & \leq \frac{2\pi(A^{i_{n}})+\pi(f_{n}(B^{i_{n}}))+2N}{n}                                            \\
    \nonumber                                  & \leq \frac{4\pi(A^{i_{n}-1})+2\pi(B^{i_{n}-1})+4N}{n}+\frac{\pi(f_{n}(B^{i_{n}}))}{n}+\frac{2N}{n} \\
                                               & \leq \frac{4}{i_{n}-1}+\frac{3N}{i_{n}-1}+\frac{6N}{n}.
  \end{align}
  Here, we use estimate \eqref{eqn-b} and the assmuption \[n\geq |A^{i_{n}}|\geq |B^{i_{n}-1}|=(i_{n}-1)|A^{i_{n}-1}|^{2}.\]
  Following the same line, we can get the same estimate for \(\pi(l_{n}(A^{i_{n}}))\).
  Combining  \(\lim_{n\to\infty}i_{n}=\infty\) with \eqref{eqn-11}, we prove the claim \eqref{eqn-claim}.
  
  Now, we assume \(I\) is a word appearing in \(y\) with length \(n\).
  For \(n\) sufficiently large there always exists a largest \(i_{n}\) such that  \(A^{i_{n}}\) appears in \(I\).
  According to the construction of \(T\), we know \(I\) must appear in \(A^{i_{n}+1}B^{j_{n}+1}A^{i_{n}+1}\) for some \(j_{n}\geq i_{n}\).
  Otherwise, \(i_{n}\) will not be the largest index.
  Combining \eqref{eqn-claim} with the fact that
  \[\pi(I)\leq 2\pi(f_{n}(A^{i_{n}+1}))+2\pi(l_{n}(A^{i_{n}+1}))+\pi(f_{n}(B^{j_{n}+1}))+\pi(l_{n}(B^{j_{n}+1}))+2N,\]
  we have
  \begin{align*}
    \limsup_{n\to\infty}\sup_{|I|=n}\frac{\pi(I)}{n} & \leq \limsup_{n\to\infty}\sup_{|I|=n}\Big\{\frac{2\pi(f_{n}(A^{i_{n}+1}))+2\pi(l_{n}(A^{i_{n}+1}))}{n}+\frac{2N}{j_{n}}+\frac{2N}{n}\Big\} \\
                                                     & \leq \limsup_{n\to\infty}\sup_{|I|=n}\frac{2N}{j_{n}}.
  \end{align*}
  Although \((i_{n},j_{n})\) depends on the choice of \(I\), we still have the following estimate
  \[2|A^{j_{n}+1}|+|B^{j_{n}+1}|\geq 2|A^{i_{n}+1}|+|B^{j_{n}+1}|\geq n.\]
  This implies
  \[\lim_{n\to\infty}j_{n}=\infty.\]
  Therefore, we obtain
  \[ \limsup_{n\to\infty}\sup_{\substack{|I|=n,\\ I\text{ appears in }y }}\frac{\pi(I)}{n}=0.\]
  
  On the other hand, from the definition we know \(N(y,B_{2^{-N-1}}(z))\) has a positive upper Banach density.
  This is equivalent to
  \[ \limsup_{n\to\infty}\sup_{\substack{|I|=n,\\ I\text{ appears in }y } }\frac{\pi(I)}{n}>0,\]
  which contradicts the above estimate.
  Therefore, we prove \(\omega_{B^*}(y)\subseteq \Pi\).

  Moreover, \(y\) is recurrent since
  \[d(y,\sigma^{|A^{n}|+|B^{n}|}(y))\leq 2^{-|A^{n}|}.\]
  Also, noticing that \(A_{1}\) does not appear in \(\Pi\), we have \(y\in \omega_{\sigma}(y)\setminus\Pi \).
  Hence, we show that \(\omega_{B^{*}}(y)\subseteq \Pi\subsetneq \omega_{\sigma}(y)\)  and \(y\in Rec(\Pi)\setminus BR(\Pi)\).
\end{proof}

\begin{rem}\label{Rem-omega-lowerB}
  Assume \(\Pi\) has two distinct periodic orbits.
  From Lemma \ref{Lem-omegalimit}, we know \[\omega_{B_{*}}(Tx)=\bigcap_{\mu\in M_{\sigma}(\Lambda)}S_{\mu}\subseteq\bigcap_{\mu\in M_{\sigma}(\Pi)}S_{\mu}=\emptyset,\]
  where \(\Lambda=\omega_{\sigma}(Tx)\).
  Therefore, we have \(\omega_{B_{*}}(Tx)=\emptyset\).
\end{rem}

\subsection{Empirical measures under \(T\)}
\label{sec-emp}
We consider the ergodic average generated by \(x\) and \(y=Tx\) for \(x\in \Pi\).
We denote their empirical measures  \(\mu_m=\frac{1}{m}\sum_{i=0}^{m-1}\delta_{\sigma^i(x)}\) and \(\nu_n=\frac{1}{n}\sum_{i=0}^{n-1}\delta_{\sigma^i(y)}\) by \(\mu_{m}\) and \(\nu_{n}\) respectively.
We denote all the real-valued continuous functions on \(\Sigma\) by \(C(\Sigma)\) and the uniform norm by \(\|\cdot\|_{\infty}\) .
For any function \(F\in C(\Sigma)\), we denote its modulus of continuity by
\[s(F,\varepsilon):=\sup_{d(x,y)\leq \varepsilon}\left|F(x)-F(y)\right|.\]
\begin{lem}
  \label{lem-m}
  For any \(x\in \Pi\) and  \(F\in C(\Sigma)\), we have
  \[\lim_{m\to\infty}\mu_{m}(F)-\nu_{\theta(m)}(F)=0.\]
  Here, \(\theta(m)=|A^{m+1}|\).
\end{lem}
\begin{proof}
  We will frequently use the identity \[|A^{m+1}|=2|A^{m}|+|B^{m}|=2|A^{m}|+m|A^{m}|^{2}.\]
  By definition, we have
  \begin{align*}
     & \quad\,\,\Big|\int_{\Sigma}F\dif\mu_m-\int_{\Sigma}F\dif\nu_{\theta(m)}\Big|                                                                               \\
     & =\Big|\frac{1}{m}\sum_{i=0}^{m-1}F(\sigma^i(x))-\frac{1}{\theta(m)}\sum_{i=0}^{\theta(m)-1}F(\sigma^i(y))\Big|                                             \\
     & \leq\Bigg|\frac{1}{\theta(m)}\sum_{i=0}^{\left|A^{m}\right|-1}F(\sigma^i(y))+\frac{1}{\theta(m)}\sum_{i=|A^{m}|+|B^{m}|}^{\theta(m)-1}F(\sigma^i(y))\Bigg| \\
     & \quad+\Bigg|\frac{1}{m}\sum_{i=0}^{m-1}F(\sigma^i(x))-\frac{1}{\theta(m)}\sum_{i=\left|A^{m}\right|}^{\left|A^{m}\right|+|B^{m}|-1}F(\sigma^i(y))\Bigg|    \\
     & := \cJ_{1}+\cJ_{2}.
  \end{align*}
  For \(\cJ_{1}\), we obtain
  \begin{align}
    \label{eqn-j1}
    \cJ_{1}\leq \frac{2|A^{m}|}{n}\|F\|_{\infty}\leq \frac{2}{m}\|F\|_{\infty}.
  \end{align}
  For any positive integer \(i\), we define \(r(i):=i-m\lfloor i/m\rfloor\).
  Notice the fact that
  \begin{equation*}
    |F(\sigma^{r(i-|A^{m}|)}(x))-F(\sigma^i(y))|\leq s(F,2^{m-r(i)})
  \end{equation*}
  for any \(|A^{m}|\leq i<|A^{m}|+|B^{m}|\).
  Hence, for \(\cJ_{2}\) we obtain
  \begin{align*}
    \cJ_{2} & \leq \Bigg|\frac{1}{m}\sum_{i=0}^{m-1}F(\sigma^i(x))-\frac{|B^{m}|}{m\theta(m)}\sum_{i=0}^{m-1}F(\sigma^i(x))\Bigg|                                   \\
            & \quad + \Bigg|\frac{1}{\theta(m)}\sum_{i=\left|A^{m}\right|}^{\left|A^{m}\right|+|B^{m}|-1}\big[F(\sigma^{r(i-|A^{m}|)}(x))-F(\sigma^i(y))\big]\Bigg| \\
            & \leq \left(1-\frac{|B^{m}|}{\theta(m)}\right) \|F\|_{\infty}+\frac{|A^{m}|^{2}}{\theta(m)}\sum_{i=1}^{m}s(F,2^{-i})                                   \\
            & \leq \frac{2}{m} \|F\|_{\infty} +\frac{1}{m}\sum_{i=1}^{m}s(F,2^{-i}).
  \end{align*}
  Combining the above estimate with \eqref{eqn-j1}, we show
  \begin{align*}
    \limsup_{m\to\infty}|\mu_{m}(F)-\nu_{\theta(m)}(F)|\leq \limsup_{m\to\infty} \Big[\frac{4}{m} \|F\|_{\infty} +\frac{1}{m}\sum_{i=1}^{m}s(F,2^{-i})\Big]=0.
  \end{align*}
\end{proof}

\begin{lem}
  \label{lem-n}
  For any \(x\in \Pi\) and  \(F\in C(\Sigma)\), we have
  \[\lim_{n\to\infty}\mu_{\tau(n)}(F)-\nu_{n}(F)=0.\]
  Here, for \(n\geq |A^{2}|\), \(\tau(n)\) is the unique integer satisfying \(|A^{\tau(n)+1}|<n\leq|A^{\tau(n)+2}|\).
\end{lem}

\begin{proof}
  Let \(\lambda=|A^{\tau(n)+1}|/n\).
  By definition, we have
  \begin{align*}
     & \quad\,\,\Big|\int_{\Sigma}F\dif\mu_{\tau(n)}-\int_{\Sigma}F\dif\nu_n\Big|                                                                                                                             \\
     & =\lambda\Big|\int_{\Sigma}F\dif\mu_{\tau(n)}-\int_{\Sigma}F\dif\nu_{|A^{\tau(n)+1}|}\Big|+\Big|(1-\lambda)\int_{\Sigma}F\dif\mu_{\tau(n)}-\frac{1}{n}\sum_{i=|A^{\tau(n)+1}|}^{n-1}F(\sigma^i(y))\Big| \\
     & := \cJ_{1}+\cJ_{2}.
  \end{align*}
  By Lemma \ref{lem-m}, we obtain
  \begin{equation*}
    \lim_{n\to\infty}\cJ_{1}=0.
  \end{equation*}
  We set \(m=\tau(n)\).
  Recall \(r(i):=i-m\lfloor i/m\rfloor\).
  For \(\cJ_{2}\), we have
  \begin{align}
    \label{est-1}
    \nonumber\cJ_{2} & =\Bigg| \frac{n-|A^{m+1}|}{nm}\sum_{i=0}^{m-1}F(\sigma^{i}(x)) - \frac{1}{n}\sum_{i=|A^{m+1}|}^{n-1}F(\sigma^i(y))\Bigg|                    \\
    \nonumber        & \leq \Bigg| \frac{n-|A^{m+1}|}{nm}\sum_{i=0}^{m-1}F(\sigma^{i}(x)) - \frac{1}{n}\sum_{i=|A^{m+1}|}^{n-1}F(\sigma^{r(i-|A^{m+1}|)}(x))\Bigg| \\
                     & \quad +\frac{1}{n}\sum_{i=|A^{m+1}|}^{n-1}\Big|F(\sigma^{r(i-|A^{m+1}|)}(x))-F(\sigma^{i}(y))\Big|.
  \end{align}
  Notice that
  \begin{align}
    \label{est-2}
     & \quad\,\,\Bigg| \frac{n-|A^{m+1}|}{nm}\sum_{i=0}^{m-1}F(\sigma^{i}(x)) - \frac{1}{n}\sum_{i=|A^{m+1}|}^{n-1}F(\sigma^{r(i-|A^{m+1}|)}(x))\Bigg|\nonumber                  \\
     & =\Bigg| \frac{n-|A^{m+1}|}{nm}\sum_{i=0}^{m-1}F(\sigma^{i}(x)) - \frac{1}{n}\sum_{i=0}^{m-1} \Big\lceil \frac{n-|A^{m+1}|-i}{m}\Big\rceil F(\sigma^{i}(x))\Bigg|\nonumber \\
     & \leq \frac{1}{n}\sum_{i=0}^{m-1}\left|\frac{n-|A^{m+1}|}{m}-\Big\lceil \frac{n-|A^{m+1}|-i}{m}\Big\rceil \right|\|F\|_{\infty}\nonumber                                   \\
     & \leq \frac{2m}{n}\|F\|_{\infty}.
  \end{align}
  Also, we have
  \begin{align}
    \label{est-3}
     & \quad\,\,\frac{1}{n}\sum_{i=|A^{m+1}|}^{n-1}\Big|F(\sigma^{r(i-|A^{m+1}|)}(x))-F(\sigma^{i}(y))\Big|\nonumber                                         \\
     & \leq \frac{1}{n} \Big\lceil \frac{n-|A^{m+1}|}{m}\Big\rceil \sum_{i=1}^{m}s(F,2^{-i})\nonumber                                                        \\
     & \quad +\chi_{n>\{|A^{m+1}|+|B^{m+1}|\}}\frac{1}{n}\sum_{i=|A^{m+1}|+|B^{m+1}|}^{n-1}\Big|F(\sigma^{r(i-|A^{m+1}|)}(x))-F(\sigma^{i}(y))\Big|\nonumber \\
     & \leq \frac{2}{m}\sum_{i=1}^{m}s(F,2^{-i})+\frac{2}{|A^{m+1}|}\|F\|_{\infty}.
  \end{align}
  Plugging above estimates into \eqref{est-1}, we obtain
  \begin{equation*}
    \lim_{n\to\infty}\cJ_{2}=0.
  \end{equation*}
  Therefore,
  \[\lim_{n\to\infty}\mu_{\tau(n)}(F)-\nu_{n}(F)=0.\]
\end{proof}

\begin{prop}\label{prop-measure}
  For any \(x\in \Pi\), we have \(V_{\sigma}(x)=V_{\sigma}(Tx)\).
  Furthermore, we have \(\omega_{\overline{d}}(x)=\omega_{\overline{d}}(Tx)\) and  \(\omega_{\underline{d}}(x)=\omega_{\underline{d}}(Tx)\).
\end{prop}

\begin{proof}
  Let \(\mu\in V_{\sigma}(x)\).
  Then, there exists a sequence of empirical measures \(\{\mu_{m_{k}}\}\) such that \(\mu_{m_{k}}\) converges to \(\mu\) weakly.
  By Lemma \ref{lem-m}, we know  \(\nu_{|A^{m_{k}+1}|}\) also converges to  \(\mu\) weakly.
  This implies \(V_{\sigma}(x)\subseteq V_{\sigma}(Tx)\).
  Similarly, by Lemma \ref{lem-n}, we have \(V_{\sigma}(Tx)\subseteq V_{\sigma}(x)\).
  
  By Lemma \ref{Lem-omegalimit}, we have \(\omega_{\xi}(x)=\omega_{\xi}(Tx)\) for any \(x\in\Pi\) and \(\xi=\overline{d},\,\underline{d}\).
\end{proof}

\subsection{DC1 pairs  under \(T\)}
We show that \(T\) keeps distributional chaos invariant.
\begin{prop}
  \label{prop-dc1}
  Assume \(x,\, \tilde{x}\in\Pi\) is a DC1-scrambled pair.
  We have  \(Tx\) and \(T\tilde{x}\) are also DC1-scrambled.
\end{prop}

\begin{proof}
  As before, we denote \(Tx\) and \(T\tilde{x}\) by \(y\) and \(\tilde{y}\) respectively.
  To show \(y\) and \(\tilde{y}\) are DC1-scrambled, it suffices to show that for any positive integer \(N\), we have
  \[\lim_{m\to\infty} \Phi_{x\tilde{x}}^{m}(2^{-N})-\Phi_{y\tilde{y}}^{|A^{m+1}|}(2^{-N})=0.\]
  Recall
  \[\Phi_{x\tilde{x}}^{(m)}(2^{-N}):=\frac{1}{m}\left|\left\{i \,|\,  d\left(\sigma^{i}(x), \sigma^{i}(\tilde{x})\right)<2^{-N}, 0 \leq i<m\right\}\right|.\]
  Let \(r(i)=i-m\lfloor i/m\rfloor\).
  For any \(|A^{m}|\leq i<|A^{m}|+|B^{m}|\) satisfying \(r(i-|A^{m}|)\leq m-N\), we have
  \[d(\sigma^i(y),\sigma^i(\tilde{y}))<2^{-N} \iff d(\sigma^{r(i-|A^{m}|)}(x),\sigma^{r(i-|A^{m}|)}(\tilde{x}))<2^{-N}.\]
  Therefore, we obtain
  \[-\frac{N}{m} \leq \Phi_{x\tilde{x}}^{(m)}(2^{-N})- \frac{1}{|B^{m}|}\left|\left\{i\,|\,d(\sigma^i(y),\sigma^i(\tilde{y}))<2^{-N},\, |A^{m}|\leq i<|A^{m}|+|B^{m}|\right\}\right|\leq \frac{N}{m}.\]
  This implies
  \begin{align*}
     & \quad\,\limsup_{m\to\infty}\Big|\Phi_{x\tilde{x}}^{m}(2^{-N})-\Phi_{y\tilde{y}}^{|A^{m+1}|}(2^{-N})\Big|                                                                                                            \\
     & \leq \limsup_{m\to\infty} \Big(1-\frac{|B^{m}|}{|A^{m+1}|}\Big)|\Phi_{x\tilde{x}}^{m}(2^{-N})|                                                                                                                      \\
     & \quad +\limsup_{m\to\infty}\frac{|B^{m}|}{|A^{m+1}|}\left|\Phi_{x\tilde{x}}^{(m)}(2^{-N})- \frac{1}{|B^{m}|}\left|\left\{i\,|\,d(\sigma^i(y),\sigma^i(\tilde{y}))<2^{-N},\, 0\leq i<|A^{m+1}|\right\}\right|\right| \\
     & \leq \limsup_{m\to\infty} \frac{N}{m}+\frac{2|A^{m}|}{|B^{m}|}                                                                                                                                                      \\
     & =0.
  \end{align*}
  Hence, \(y\) and \(\tilde{y}\) is a DC1-scrambled pair.
\end{proof}

\subsection{Regular and irregular points under \(T\)}

\begin{prop}
  \label{prop-ergodic}
  Suppose that  \(\varphi\) is a continuous function on \(\Sigma\) and \(I_{\varphi}(\Sigma)\neq \emptyset\).
  Then for any \(a\in \operatorname{Int}(L_{\varphi}(\Sigma))\), we have \(T\) maps \(I_{\varphi}(\Sigma)\cap \Pi\) into \(I_{\varphi}(\Sigma)\) and maps \(R_{\varphi}(\Sigma,a)\cap \Pi\) into \(R_{\varphi}(\Sigma,a)\).
\end{prop}

\begin{proof}
  We use the same notations as Lemma \ref{lem-n}.
  Let \(x\in \Pi\) and \(y=Tx\).
  By Lemma \ref{lem-n}, we have
  \begin{equation}
    \lim_{n\to\infty}\left|\frac{1}{n} \sum_{i=0}^{n-1} \varphi\left(\sigma^{i} (y)\right)- \frac{1}{\tau(n)} \sum_{i=0}^{\tau(n)-1} \varphi\left(\sigma^{i} (x)\right)\right|=0.
  \end{equation}
  Therefore, we obtain
  \begin{align*}
    \limsup_{n\to\infty}\frac{1}{n}\sum_{i=0}^{n-1}\varphi(\sigma^i(y)) & =\limsup_{n\to\infty}\frac{1}{\tau(n)}\sum_{i=0}^{\tau(n)-1}\varphi(\sigma^i(x)) \\
                                                                        & =\limsup_{m\to\infty}\frac{1}{m}\sum_{i=0}^{m-1}\varphi(\sigma^i(x)),
  \end{align*}
  and
  \begin{align*}
    \liminf_{n\to\infty}\frac{1}{n}\sum_{i=0}^{n-1}\varphi(\sigma^i(y)) & =\liminf_{n\to\infty}\frac{1}{\tau(n)}\sum_{i=0}^{\tau(n)-1}\varphi(\sigma^i(x)) \\
                                                                        & =\liminf_{m\to\infty}\frac{1}{m}\sum_{i=0}^{m-1}\varphi(\sigma^i(x)).
  \end{align*}
  This implies \(T\) maps \(I_{\varphi}(\Sigma)\cap \Pi\) into \(I_{\varphi}(\Sigma)\) and maps \(R_{\varphi}(\Sigma,a)\cap \Pi\) into \(R_{\varphi}(\Sigma,a)\).
\end{proof}

\section{Construction of an auxiliary subshift \(\Pi\)}
\label{sec-subshift}
In this section, we construct a specific mixing subshift \(\Pi\) which will be used in the proof of main theorems.

For any \(l\geq 3\) and \(k\geq2\), let
\(\B=\left\{(0\cdots0),\,\cdots,\,(k-1\cdots k-1) \right\},\)
where \((i\cdots i)\) has length \(l\) for any \(0\leq i<k\).
We define a subshift on \(k\) symbols \(\Pi\subseteq\Sigma\) as
\[\Pi=\{x=(x_{1}x_{2}\cdots)\in\Sigma\,|\,(x_{i+1}\cdots x_{i+l})\notin\B \text{ for any } i\geq0\}.\]
In other words, any symbol does not appear consecutively \(l\) times in \(\Pi\).

\begin{prop}
  \label{prop-Pi-L}
  For any \(l\geq 3\), we have \((\Pi,\sigma)\) is topological mixing.
  Moreover, \((\Pi,\sigma)\) satisfies the specification property.
\end{prop}
\begin{proof}
  Let \(U,\,V\) be two nonempty open sets in \(\Pi\).
  Assume \(x\in U\) and \(y\in V\), then there exists an integer \(n\) such that  \([X_{n}]\cap \Pi\subseteq U\) and \([Y_{n}]\cap\Pi\subseteq V\).
  For any \(m\geq 3\), we construct a word \(Z_{m}=(z_{1}\cdots z_{m})\) as follows.
  We let \(z_{1}\neq x_{n}\) and \(z_{m}\neq y_{1}\) since \(k\geq 2\).
  Moreover, we may assume \(Z_{m}\) appears in \(\Pi\) since \(l\geq 3\).
  Hence, we obtain
  \[[X_{n}Z_{m}Y_{n}]\cap \Pi \subseteq U.\]
  This implies for any \(m\geq 3\), we have
  \[\sigma^{n+m}(U)\cap V\supseteq [Y_{n}]\cap \Pi\neq \emptyset.\]
  Therefore, \(\Pi\) is topological mixing.
  From \cite[Proposition 21.2]{DGS} we know every topological mixing subshift satisfies the specification property.
  Hence, \(\Pi\) satisfies the specification property.
\end{proof}

\begin{lem}
  \label{Lem-Pi-L}
  Assume \(x_{1},\cdots,x_{n}\in\Sigma\) are periodic but not fixed.
  Then there exists \(l\geq 3\) such that \(x_{i}\in\Pi\) for any \(1\leq i\leq n\).
\end{lem}
\begin{proof}
  Assume \(p_{i}\geq 2\) is the period of \(x_{i}\) for \(1\leq i\leq n\).
  Let \(l\geq \max\left\{p_{1},\,p_{2},\,\cdots,\,p_{n} \right\}\).
  Since \(x_{i}\) is periodic and not fixed, there is no symbol appearing consecutively \(p_{i}\) times in \(x_{i}\).
  Hence, by the definition of \(\Pi\), we have \(x_{i}\in\Pi\) for any \(1\leq i\leq n\).
\end{proof}

\section{Proof of main theorems}
\label{sec-proof}
The following Lemma is an application of  \cite[Theorem 1.6, Theorem 6.1]{CT2019}.
\begin{lem}
  \label{Lem-casei}
  Assume \(l\geq 3\) and \(k\geq 2\) and let \((\Pi,\sigma)\) be the subshift defined in the previous section.
  Then \(\left\{ x\in Trans(\Pi)\,|\, x \text{ satisfies Case (i) }\right\}\) contains an uncountable DC1-scrambled subset for i=1-6.
  Furthermore, if \(\varphi\) is a continuous function on \(\Pi\) and \(I_{\varphi}(\Pi) \neq \emptyset \) then for any \(a \in \operatorname{Int}\left(L_{\varphi}(\Pi)\right)\), the level set \(\{x \in Trans(\Pi) \,|\, x \text { satisfies Case }(i)\}\) contains an uncountable DC1-scrambled subset in  \( I_{\varphi}(\Pi)\),  \( R_{\varphi}(\Pi)\) and  \(R_{\varphi}(\Pi,a)\), respectively, for i=1-6.
\end{lem}
\begin{proof}
  We only need to  check the conditions in \cite[Theorem 1.6, Theorem 6.1]{CT2019}.
  By Proposition \eqref{prop-Pi-L}, we know \(\Pi\) satisfies the specification property.
  
  For Case (2)-(6), it suffices to show that \(\Pi\)  is not uniquely ergodic by \cite[Theorem 1.6]{CT2019}.
  This is because \(\Pi\) has two different periodic orbits.
  
  For  Case (1), from \cite[Theorem 6.1, Remark 6.2, Remark 6.3]{CT2019} we only need to show that there exist an integer \(m\) such that \(\Pi\) does not contain periodic points with period \(m\).
  Notice the fact \(\Pi\) does not contain any fixed point.
  Therefore, we also show Case (1).
\end{proof}

\begin{lem}
  \label{Lem-omegalimit-T}
  Assume \(l\geq 3\) and \(k\geq 2\) and let \((\Pi,\sigma)\) be the subshift defined in the previous section.
  For any \(x\in Trans(\Pi)\) satisfying Case (i), we have \(Tx\in Rec(\Sigma)\setminus BR(\Sigma)\) and \(Tx\) satisfies Case (i') for i=1-6.
\end{lem}
\begin{proof}
  For \(\xi=\overline{d},\,\underline{d}\), we obtain \(\omega_{\xi}(x)=\omega_{\xi}(Tx)\) by Proposition \ref{prop-measure}.
  Since \(x\in Trans(\Pi)\) and satisfies Case (i), we have \(\omega_{B^*}(x)=\omega_{\sigma}(x)=\Pi\).
  Therefore, by Proposition \ref{prop-BR} we know \(\omega_{B^*}(x)=\omega_{B^*}(Tx)=\Pi\) and \(Tx\in Rec(\Sigma)\setminus BR(\Sigma)\).
  By Remark \ref{Rem-omega-lowerB}, we also have \(\omega_{B_*}(x)=\omega_{B_*}(Tx)=\emptyset\).
  From the above discussion, we end up showing that \(Tx\) satisfies Case (i').
\end{proof}

Now, we give the proof of main theorems.
\begin{proof}[Proof of Theorem \ref{Thm-main-omega-DC1}]
  For the density, we notice that \(A_{1}\) can be chosen as any word not appearing in \(\Pi\).
  For any nonempty open set \(U\subseteq\Sigma\), we choose \(A_{1}\) such that \([A_{1}]\subseteq U\) and \(A_{1}\) does not appear in \(\Pi\).
  By Lemma \ref{Lem-casei}, there exists \(x\in Trans(\Pi)\) satisfying Case (i).
  Therefore, by Lemma \ref{Lem-omegalimit-T}, we have \(Tx\) satisfies Case (i') and \(Tx\in Rec(\Sigma)\setminus BR(\Sigma)\).
  Since \(U\) is arbitrary, we show that  \(\left\{ x\in Rec(\Sigma)\setminus BR(\Sigma) \,|\,x \text{ satisfies Case (i') }\right\}\) is dense in \(\Sigma\).
  
  For distributional chaos, from Lemma \ref{Lem-casei}, there exists an uncountable DC1-scrambled set \(DC_{i}\) contained in \(\{x \in Trans(\Pi) \,|\, x \text{ satisfies Case }(i)\}\).
  By Proposition \ref{prop-dc1} and the fact \(T\) is an injection, we know \(T(DC_{i})\) is also an uncountable DC1-scrambled set.
  Moreover, by Lemma \ref{Lem-omegalimit-T}, we have \(T(DC_{i}) \subseteq\left\{ x\in Rec(\Sigma)\setminus BR(\Sigma) \,|\,x \text{ satisfies Case (i') }\right\}\).
\end{proof}

\begin{proof}[Proof of Theroem \ref{Thm-multifraction}]
  For the irregular points part, we only need to notice that \(I_{\varphi}(\Pi)=I_{\varphi}(\Sigma)\cap \Pi\).
  Then, by Proposition \ref{prop-ergodic} and Lemma \ref{Lem-casei}, we can use the same argument as in the proof of Theorem \ref{Thm-main-omega-DC1}.
  
  For the regular level set, we show that for any \(a\in\operatorname{Int}(L_{\varphi}(\Sigma))\) there exists \(l\geq 3\) such that \(R_{\varphi}(\Pi,a)=R_{\varphi}(\Sigma,a)\cap \Pi\).
  By definition, there exists \(\mu_{1},\,\mu_{2}\in M_{\sigma}(\Sigma)\) such that \[\int_{\Sigma}\varphi\dif\mu_{1}<a<\int_{\Sigma}\varphi\dif\mu_{2}.\]
  By \cite[Proposition 21.8]{DGS}, we may assume \(\mu_{1}\) and \(\mu_{2}\) are periodic measures, i.e., \(S_{\mu_{i}}\) is the orbit of some periodic point.
  Moreover, we may assume  \(\mu_{1}\) and \(\mu_{2}\) are not atomic.
  By Lemma \ref{Lem-Pi-L}, there exists \(l\geq 3\) such that \(S_{\mu_{i}}\subseteq \Pi\).
  Then, we derive \(a\in\operatorname{Int}(L_{\varphi}(\Pi))\) which implies \(R_{\varphi}(\Pi,a)=R_{\varphi}(\Sigma,a)\cap \Pi\).
  Again, we apply the argument in the proof of Theorem \ref{Thm-main-omega-DC1} and finish the proof.
\end{proof}

\section{Other dynamical systems}

\subsection{Two-sided full shifts}
We extend the results in previous sections to two-sided full shifts.
Recall that \(\mathcal{A}=\left\{0,1,\cdots,k-1\right\}\) is a set of \(k\) symbols.
With abuse of notations, we also use \((\Sigma,\sigma)\) to denote the two-sided full shift on \(k\) symbols.
Here, \(\Sigma=\prod_{n=-\infty}^{+\infty}\mathcal{A}\) and for any \(a\in\Sigma\) we denote the \(n\)-th coordinate of \(a\) by \(a_{n}\).
Then \(\Sigma\) is a compact space equipped with the following metric
\[d(a,b):=\sum_{n=1}^{\infty}\frac{\delta(a_n,b_n)+\delta(a_{-n},b_{-n})}{2^n},\  \text{for any } a,\, b\in \Sigma.\]
The shift operator \(\sigma:\Sigma\to\Sigma\) is defined by
\[\sigma(a)_{n}=a_{n+1},\  \text{for any }a\in \Sigma.\]

For any \(a\in\Sigma\) and \(n\leq m\), let \(A_{[n,m]}\) be the word \((a_{n}\cdots a_{m})\).
For any  word \(A\) and \(n\leq m\), let \([A]_{n,m}\) denote the cylinder set \(\{x\in\Sigma|A=X_{[n,m]}\}.\)
We say \((\Pi,\sigma)\) is a subshift if \(\Pi\) is closed and invariant under \(\sigma\).

\begin{Thm}\label{thm-twosided}
  Given \(k\geq 2\), let \((\Sigma,\sigma)\) denote the two-sided full shift on \(k\) symbols.
  Then the gap-set  \((Rec(\Sigma)\setminus BR(\Sigma))\)   contains an uncountable
  DC-1 scrambled subset.
\end{Thm}
\begin{proof}
  The key idea is again to construct an injection \(T\) from a proper subshif \(\Pi\) to \(\Sigma\).
  Now, we define \(T\) by induction.
  Let \(A^{1}\) be a word not appearing in \(\Pi\).
  For any \(x\in\Sigma\), we define
  \[B^{n}=X_{[-\lfloor \sqrt{n}\rfloor,n]}\cdots X_{[-\lfloor \sqrt{n}\rfloor,n]} \quad \text{and} \quad A^{n+1}=A^{n}B^{n}A^{n}0A^{n}.\]
  Then \(Tx\) is defined by
  \[Tx:=\lim_{n\to\infty}[A^{n}0A^{n}]_{-|A^{n}|,|A^{n}|}.\]
  We claim without proof that \(T\) maps \(Trans(\Pi)\) to \(Rec(\Sigma)\setminus BR(\Sigma)\) and keeps DC1 pair invariant.
  
  Recall \(\B=\left\{(0\cdots0),\,\cdots,\,(k-1\cdots k-1) \right\},\) where \((i\cdots i)\) has length \(l\) for any \(0\leq i<k\).
  We define subshift \(\Pi\subseteq\Sigma\) as
  \[\Pi=\{x\in\Sigma\,|\,(x_{i+1}\cdots x_{i+l})\notin\B \text{ for any } i\in \mathbb{Z}\}.\]
  Then, we can show that \(\Pi\) satisfies the specification property and is not uniquely ergodic.
  Therefore, we use the same argument in the previous sections to complete the proof.
\end{proof}

\subsection{Expansive systems with the shadowing property}
We recall the definitions of expansive systems and the shadowing property, for example, see \cite{sakai2003various}.
A dynamical system \((X,f)\) is \textit{positively expansive} if there exists \(c>0\) such that \(d(f^{i}(x),f^{i}(y))<c\) for all \(i\geq 0\) implies \(x=y\).
A dynamical system \((X,f)\) has \textit{the shadowing property} if for any \(\varepsilon>0\), there exists \(\delta>0\) such that for any \(\delta\)-pseudo-orbit \(\{x_{i}\}_{i=0}^{\infty}\), there exits \(y\in X\) satisfying \(d(x_{i},f^{i}(y))<\varepsilon\) for all \(i\geq 0\).
Here, as usual, a \(\delta\)\textit{-pseudo-orbit} is a sequence of points \(\{x_{i}\}_{i=0}^{\infty}\) in \(X\) satisfying \(d(f(x_{i}),x_{i+1})<\delta\) for all \(i\geq 0\).

Then, we extend results to expansive systems with the shadowing property as follows.
\begin{maintheorem}
  \label{thm-shadow}
  Suppose \((X,  f)\) is a topological dynamical system with the shadowing property.
  If  \((X,  f)\) is a positively expansive map or is an expansive homeomorphism,  then \(Rec(X)\setminus BR(X)\)  contains an uncountable DC-1 scrambled subset.
\end{maintheorem}
\begin{rem}
  In differential dynamics, both expanding systems and hyperbolic systems satisfy the shadowing property and expansiveness, for example, see \cite{pilyugin2006shadowing} for details. 
  Therefore, in such a system $Rec\setminus BR$ also contains an uncountable DC-1 scrambled subset.
\end{rem}

We first state a characterization for systems with the shadowing property.
This lemma enables us to make connection to symbolic dynamics.
\begin{lem}
  \label{horseshoe}
  \cite{DOT}
  Suppose \((X,  f)\) is a topological dynamical system.
  If \((X,  f)\) has the shadowing property and its topological entropy \(h_{top}(f)>0\), then for any \(0<\alpha<h_{top}(f)\) there are  integers  \(k,\,n\),   \(\log (k)/n>\alpha\) and a closed set \(\Delta\subseteq X\)
  invariant under $f^n$ such that there is a factor map \(\pi:(\Delta,   f^n)\to (\Sigma,  \sigma)\) onto the one-sided full shift on \(k\) symbols.
  If in addition, \((X,  f)\) is positively expansive,   then \(\pi\) is a conjugacy.
\end{lem}

The following observation on recurrent property has been proved by Zhou for \(W\) and \(QW\).
Here, we give a proof for \(Rec\) and \(BR\).
Henceforth, we will use $Rec(X,f)$ and	$BR(X,f)$ to re-denote the sets of recurrent points and Banach recurrent points for emphasizing the dynamic $f$.

\begin{prop}
  \label{prop-fr-recurrent}
  Suppose \((X,f)\) is a topological dynamical system.
  Then \(Rec(X,f)=Rec(X,f^{n})\) and \(BR(X,f)=BR(X,f^{n})\) for any \(n\in\mathbb{N^{+}}\).
\end{prop}

\begin{proof}
  For the simplicity of notations, we let \(g=f^{n}\).
  By definitions, it is easy to see that \(Rec(X,g)\subseteq Rec(X,f)\) and \(BR(x,g)\subseteq BR(X,f)\).
  Now we show the other direction is also true.
  We fix \(x\in Rec(X,f)\).
  Let \(\Lambda=\omega_{f}(x)\) and \(\Lambda_{i}=\omega_{g}(f^{i}(x))\) for \(0\leq i<n\).
  For convenience, we also extend the definition to all nonnegative integers by \(\Lambda_{n+i}=\Lambda_{i}\).
  Since \(x\in\Lambda\), for any \(\varepsilon>0\) we have
  \[N_{f}(x,B_{\varepsilon}(x))=\bigcup_{i=0}^{n-1}N_{g}(f^{i}(x),B_{\varepsilon}(x))\neq \emptyset.\]
  This implies
  \begin{equation}
    \label{eqn-lambda}
    \Lambda=\bigcup_{i=0}^{n-1}\Lambda_{i}.
  \end{equation}
  Without loss of generality, we assume \(m\) is the smallest nonnegative integer such that \(x\in\Lambda_{m}\).
  Since \(f^{m}(\Lambda_{m})\subseteq \Lambda_{2m}\), we have \(f^{m}(x)\in \Lambda_{2m}\).
  This implies
  \[\Lambda_{m}=\omega_{g}(f^{m}(x))\subseteq \Lambda_{2m}.\]
  In particular, we show \(x\in\Lambda_{2m}\).
  Repeating the same argument, we obtain \(x\in\Lambda_{lm}\) for any positive integer \(l\).
  Therefore, we derive \(m=0\) since \(m\) is the smallest such nonnegative integer.
  This shows that \(x\in\omega_{g}(x)\) and thus \(Rec(X,f)\subseteq Rec(X,g)\).
  
  Now, we further assume \(x\in BR(X,f)\).
  Notice again that for \(0\leq i<n\), we have
  \[f(\Lambda_{i})\subseteq \Lambda_{i+1}.\]
  On the other hand, we know that
  \[f^{n}(\Lambda_{0})=g(\omega_{g}(x))=\omega_{g}(x)=\Lambda_{0}.\]
  Therefore, we obtain
  \begin{equation}
    \label{eqn-flambda}
    f(\Lambda_{i})=\Lambda_{i+1}\quad \text{and} \quad f^{-1}(\Lambda_{i+1})=\Lambda_{i}.
  \end{equation}
  By Lemma \ref{Lem-omegalimit}, there exists \(\mu\in M_{f}(\Lambda)\) such that \(x\in S_{\mu}\).
  From \eqref{eqn-lambda} and \eqref{eqn-flambda}, we know that \(\mu(\Lambda_{0})>0\).
  Therefore, we can define a measure \(\nu\in M_{g}(\Lambda_{0})\) by
  \[\nu=\frac{1}{\mu(\Lambda_{0})}\mu|_{\Lambda_{0}}.\]
  If \(x\) is only contained in \(\Lambda_{0}\), then \(x\in S_{\nu}\) since \(\Lambda_{i}\) is closed.
  Otherwise, there is a smallest integer \(p>0\) such that \(x\) is  contained in \(\Lambda_{p}\).
  As previous discussion, we obtain \(n=pq\) for some  integer \(q\).
  Then, we obtain
  \[\Lambda_{0}\subseteq f^{p}(\Lambda_{0})\subseteq \cdots \subseteq f^{pq}(\Lambda_{0})=\Lambda_{0}.\]
  By \eqref{eqn-flambda} we have \(\Lambda_{0}=\Lambda_{lp}\) for any \(l\geq0\).
  Again, we get \(x\in S_{\nu}\) since \(\Lambda_{i}\) is closed.
  This implies \(x\in BR(X,g)\) and thus \(BR(X,f)\subseteq BR(X,g)\).
  
\end{proof}

\begin{prop}
  \label{dc1-fr}
  Suppose \((X,  f)\) is a topological dynamical system and \(n\in\N^{+}\).
  If   \((x, y) \in X \times X\) is   a  distributionally chaotic pair under \(f^{n}\), then
  \((x,y) \) is  a  distributionally chaotic  pair under $f$.
\end{prop}
\begin{proof}
  It comes directly from the uniform continuity of  $f,\cdots, f^{n-1}$ and the definition of  distributionally chaotic of type 1.
\end{proof}

\begin{proof}[Theorem \ref{thm-shadow}]
  We take \(\Delta,\Sigma\) and \(n\) as Lemma \ref{horseshoe}.
  Note that by Proposition \ref{prop-fr-recurrent}
  \begin{align*}Rec(\Sigma,\sigma)\setminus BR(\Sigma,\sigma)\simeq Rec(\Delta,f^n)\setminus BR(\Delta,f^n) \\
    \subseteq Rec(X,f^n)\setminus BR(X,f^n)=Rec(X,f)\setminus BR(X,f).
  \end{align*}
  Here, \(\simeq\) denotes the conjugacy.
  Thus by Theorem \ref{Thm-main-symbolic} and Proposition \ref{dc1-fr}  $Rec(X,f)\setminus BR(X,f)$  contains an uncountable DC-1 scrambled subset.
  
  Following the proof in \cite{DOT}, we can also have parallel results of Lemma \ref{horseshoe} for the homeomorphism case.
  Then, similar to the above arguments we can replace Theorem \ref{Thm-main-symbolic} by Theorem \ref{thm-twosided} to end the proof.
\end{proof}

\subsection{Non-uniformly hyperbolic systems}
Let $f$ be a $C^{1}$ diffeomorphism over a compact Riemannian manifold $M.$
An ergodic measure is called {\it hyperbolic} if all its Lyapunov exponents are nonzero.
It was proved in \cite{Katok,Katok1} that for a   $C^{1+\alpha}$ diffeomorphism $f$ with a nontrivial hyperbolic ergodic measure, there exists a subset $\Delta\subseteq M$ and $n\geq1, \,k\geq 2$ such that the subsystem $(\Delta,f^n)$ is topologically conjugated to the two-sided full shift on \(k\) symbols. 
Thus by  Theorem \ref{thm-twosided} one has

\begin{maintheorem}
  Let $M$ be a compact Riemannian manifold  of dimension at least $2$ and $f$ be a $C^{1+\alpha}$ diffeomorphism.
  If $\mu$ is a nontrivial hyperbolic ergodic measure,  then
  $Rec(M)\setminus BR(M)$  contains an uncountable
  DC-1 scrambled subset.
\end{maintheorem}

\begin{rem}
  It is still unknown whether Theorem \ref{Thm-main-omega-DC1} also holds for systems in this section.
  One main difficulty is that we do not know the relation of Case (i') for $f$ and Case (i') for $f^n.$
  It is also unknown whether one can use the specification property to give a direct construction instead of using the technique of Petersen, He and Zhou.
\end{rem}

\medskip
\section*{acknowledgements}
The authors would like to thank Xiaobo Hou for some comments and Prof. Xiaoyi Wang for providing some references to us.

\medskip
\bibliographystyle{plain}      
\bibliography{mybib.bib}   

\end{document}